\documentclass[letter,11pt]{article}

\newcounter{defn}

\usepackage{amsmath}

\usepackage{comment}
\usepackage[linktocpage=true]{hyperref}
\usepackage[margin=1.2 in, top=1 in, bottom= 1.2 in]{geometry}

\usepackage[english]{babel}
\usepackage{amsfonts}
\usepackage{mathrsfs}
\usepackage[mathscr]{euscript}
\usepackage{bbm}
\usepackage{latexsym}
\usepackage{math dots}
\usepackage{amssymb}
\usepackage{mathtools}
\usepackage{graphicx}
\usepackage{tikz}

\usepackage{enumitem}
\setlist{nolistsep}
\usepackage{amsthm}
\usepackage{relsize}
\usepackage{nicefrac}
\usepackage[capitalize]{cleveref}

\newtheoremstyle{plain}{3mm}{3mm}{\slshape}{}{\bfseries}{.}{.5em}{}
\newtheoremstyle{definition}{2mm}{2mm}{}{}{\bfseries}{.}{.5em}{}
\theoremstyle{plain}
	
\newtheorem{theorem}{Theorem}

\newtheorem{lemma}[theorem]{Lemma}

\newtheorem{corollary}[theorem]{Corollary}

\theoremstyle{definition}
\newtheorem{definition}[defn]{Definition}

\theoremstyle{plain}
\newtheorem*{namedthm}{\namedthmname}
\newcounter{namedthm}
	
\newcommand{\eps}{\epsilon}

\title{Ergodicity of a surgered flow on unit tangent bundle of hyperbolic surface}
\author{Aritro Pathak}

\begin{document}

\maketitle
\begin{abstract}
    Starting with a trivial periodic flow on $\mathbb{S}M$, the unit tangent bundle of a genus two surface, we perform a Dehn surgery on the manifold around an annulus of a curve on $\mathbb{S}M$ that projects to a self-intersecting closed geodesic on $M$, to get a surgered flow, so that the orbit saturation of the surgery annulus is ergodic with respect to the volume measure. This surgered flow projects to a map on the surgery track that can be taken to be a linked twist map with oppositely oriented shears which generates the ergodic behavior for sufficiently strong shears in the surgery. This answers a question of Heberle and Hasselblatt.
\end{abstract}

\section{Introduction}

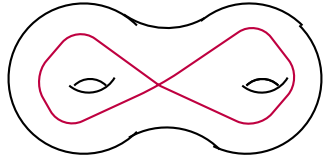
\begin{figure}
\centering
    \begin{tikzpicture}
        \draw [thick] (1,0) arc [radius=1, start angle=45, end angle= 130];
        \draw [thick] (-0.3,0.1) arc [radius=1, start angle=300, end angle= 235];
        \draw [thick] (-1.2,0) arc [radius=1, start angle=45, end angle= 315];
        \draw [thick] (-.1,-1.65) arc [radius=1, start angle=45, end angle= 120];
        \draw [thick] (0.96,0) arc [radius=1, start angle= 45 , end angle= -116];
        \draw [thick] (0.2,-0.8) arc [radius=0.4, start angle=230,end angle=330];
        \draw [thick] (0.68,-0.8) arc [radius=0.3, start angle=45,end angle=135];
        \draw [thick] (-2.1,-0.8) arc [radius=0.4, start angle=230,end angle=330];
        \draw [thick] (0.68,-0.8) arc [radius=0.3, start angle=45,end angle=135];
        \draw [thick] (-1.61,-0.8) arc [radius=0.3, start angle=45,end angle=135];
        \draw [ rounded corners=3mm, thick, purple] (-0.9,-0.8) -- (0.4,-1.4) -- (1.0,-0.7) -- (0.4,0.08) -- (-0.9,-0.8)--(-2.2,-1.4)--(-2.6,-0.7)--(-2.0,-0)--(-0.9,-0.8);
    \end{tikzpicture} 
    \caption{Self-intersecting closed geodesic $\beta:[0,1]\to M $ on the hyperbolic surface $M$ of genus 2. The surgery region in the unit tangent bundle of $M$ is described above, as shown in \cref{fig:fig1}. The surgery procedure leads to hyperbolicity and ergodicity for the part of the altered flow that intersects the surgery neighborhood. However, the altered flow is not mixing}
    \label{fig:hypsur}
\end{figure}

%To formalize this soon and eventually send to ETDS.

%This is a modification of the argument of Przytycki. In this case, we are looking at a situation where there are two symmetric squares. Note that the simple situation as in the case of Przytycki doesn't hold here because of our identifications; in Przytycki's case the only way to transport from one lobe to the other was via the central square region whereas in our identification the dynamics of course can switch from one lobe to the other without coming back to the central square region. \footnote{The only place where this symmetry really matters is when the first return to the central region intersects both squares. By extending our methods, the argument can be adopted to the non-symmetric case as well which we don't do here.} 
 
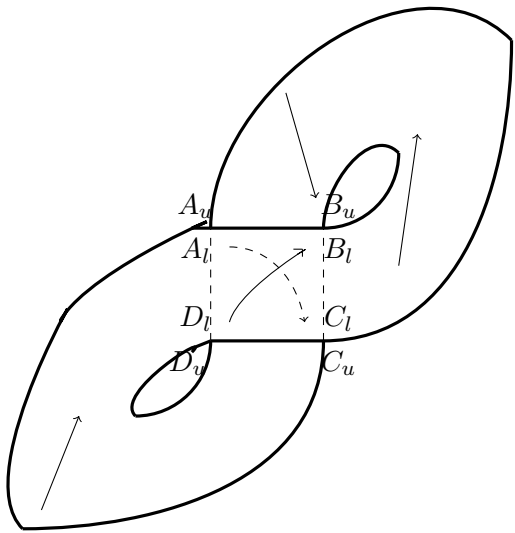
\begin{figure}
\centering
    \begin{tikzpicture}
        \draw[very thick] (0,0) to [out=00,in=270] (1,1);
        \draw[very thick] (1,1) to [out=135,in=90] (0,0);
        \draw[very thick] (0,-1.5) to [out=00,in=270] (2.5,2.5);
        \draw[very thick] (2.5,2.5) to [out=135,in=90] (-1.5,0);
        \draw[very thick] (0,0) to (-1.5,0);
        \draw[very thick] (0,-1.5) to (-1.5,-1.5);
        \draw[dashed] (-1.5,-1.5) to (-1.5,0);
        \draw[dashed] (0,-1.5) to (0,0);
        \draw[very thick] (-1.5,-1.5) to [out=270,in=00] (-2.5,-2.5);
        \draw[very thick] (-2.5,-2.5) to [out=135,in=45] (-1.75,-1.65);
        \draw[very thick] (-1.75,-1.60) to (-1.5,-1.5);
        \draw[very thick] (0,-1.5) to [out=270, in=00](-4,-4);
        \draw[very thick] (-4,-4) to [out=135, in=60](-3.5,-1.25);
        \draw[very thick] (-3.5,-1.23) to [out=60, in=23](-1.75,0);
        \draw[very thick] (-1.75,0) to (-1.5,0);
        \draw (-1.25,-1.25) to [out=75, in=31] (-0.35,-0.35);
        \draw[->] (1,-0.5) to (1.25,1.25);
        \draw[->] (-3.75,-3.75) to (-3.25,-2.50);
        \draw[->] (-0.5,1.8) to (-0.1,0.4);
        \draw (-0.28,-0.28) to (-0.40,-0.28);
        \draw (-0.28,-0.28) to (-0.28,-0.38);
        \draw[dashed, ->] (-1.25,-0.25) to [out=00,in=100] (-0.25,-1.25);
        \node [below] at (-1.7,0) {$A_l$};
        \node [above] at (-1.7,0) {$A_u$};
        \node [above] at (0.20,0) {$B_u$};
        \node [below] at (0.20,0) {$B_l$};
        \node [above] at (-1.7,-1.5) {$D_l$};
        \node [below] at (-1.8,-1.5) {$D_u$};
        \node [above] at (0.2,-1.5) {$C_l$};
        \node [below] at (0.2,-1.5) {$C_u$};
    \end{tikzpicture}
\caption{The surgery region is a two-layered track, with a bottom square region whose vertices are $A_l,B_l,C_l,D_l$, a top square region whose vertices are $A_u,B_u,C_u,D_u$, and two separate lobes that connect these regions. These two square regions lie as two stacks, one on top of the other. The map is a regular shear map on this track, as described in the text, and they are linked over this square region. The arrows show the direction of the shears in those locations. In this analysis, when projected on a plane, the bottom and the top layers of the lobe intersect perpendicularly.}
\label{fig:fig1}
\end{figure}

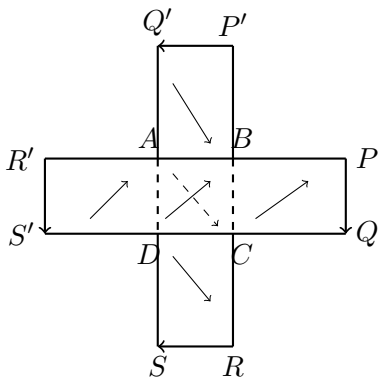
\begin{figure}
\centering
    \begin{tikzpicture}
         \draw[thick] (-2,0.5) to (2,0.5);
         \draw[thick] (-2,-0.5) to (2,-0.5);
         \draw[->,thick] (-2,0.5) to (-2,-0.5);
         \draw[->,thick] (2,0.5) to (2,-0.5);
         \draw[thick] (0.5,0.5) to (0.5,2);
         \draw[thick] (-0.5,0.5) to (-0.5,2);
         \draw[->,thick] (0.5,2) to (-0.5,2);
         \draw[thick] (0.5,-0.5) to (0.5,-2);
         \draw[thick] (-0.5,-0.5) to (-0.5,-2);
         \draw[->,thick] (0.5,-2) to (-0.5,-2);
         \draw[->] (0.8,-0.3) to (1.5,0.2);
         \draw[->] (-0.3,1.5) to (0.2,0.7);
         \draw[->] (-0.3,-0.8) to (0.2,-1.4);
         \draw[->] (-1.4,-0.3) to (-0.9,0.2);
         \draw[->] (-0.4,-0.3) to (0.2,0.2);
         \draw[->,dashed] (-0.3,0.3) to (0.3,-0.4);
         \draw[dashed, thick] (0.5,0.5) to (0.5,-0.5);
         \draw[dashed, thick] (-0.5,0.5) to (-0.5,-0.5);
         \node [above] at (-0.62,0.5) {$A$};
         \node [above] at (0.62,0.5) {$B$};
         \node [below] at (0.62,-0.5) {$C$};
         \node [below] at (-0.62,-0.5) {$D$};
         \node[right] at (2,0.5) {$P$};
         \node[right] at (2,-0.5) {$Q$};
         \node[above] at (0.5,2) {$P'$};
         \node[above] at (-0.5,2) {$Q'$};
         \node[below] at (0.5,-2) {$R$};
         \node[below] at (-0.5,-2) {$S$};
         \node[left] at (-2,-0.5) {$S'$};
         \node[left] at (-2,0.5) {$R'$};

    \end{tikzpicture}
\caption{The linked twist map description of the map from \cref{fig:fig1}. The horizontal track is the region $W=\{(0\leq x \leq 1)\times (y_0\leq y\leq y_1)\}$, the vertical track is the region $V=\{(x_0\leq x \leq x_1)\times (0\leq y\leq 1)\}$. We call $\Tilde{T}=W\cup V$. The map $H=G\cdot F$ is linked on the central square region $S$, by the shear map $F$ on $W$ and a vertical shear map $G$ on $V$. The points $A_l,A_u$ are stacked on top of one another and denoted $A$, and so forth. The pairs $(P,P'),(Q,Q'),(R,R'),(S,S')$ are identified and thus the segments $PQ,P'Q'$ are identified together and so are the segments $SR,S'R'$. The shear on the horizontal strip is denoted by $F$ and that on the vertical strip is denoted by $G$, although the underlying map $H$ is the same. Both the squares $S_1$ and $S_2$ coincide to give the central square $S$ bounded by the points $A,B,C,D$.}
\label{fig:fig2}
\end{figure}

Consider the self-intersecting closed geodesic $\beta:[0,1]\to M$ on the genus two surface $M$ as shown in \cref{fig:hypsur}. This gives an immersed submanifold of $M$ and two points $t_0,t_1 \in [0,1]$ with $\beta(t_0)=\beta(t_1)$, which is the point of intersection. 

Consider $S^{1}=[0,1]/\sim$ identifying $0$ and $1$, and any embedding $\theta:S^1 \to \mathbb{S} M$, 
with the following property: each $\theta(t)$ an element of the $\mathbb{S}^{1}$ fiber over $\beta(t)$ in $\mathbb{S}M$ when $t \notin \{t_0,t_1\}$, while $\theta(t_0) \neq \theta(t_1)$ and both $\theta(t_0),\theta(t_1)$ belong to the $\mathbb{S}^{1}$ fiber over the point $\beta(t_0)=\beta(t_1)$.

On $\mathbb{S}M$, this creates a closed curve $\theta$ which avoids intersecting itself and projects under the canonical projection $\pi:\mathbb{S}M\to M$ to the curve $\beta$. Consider an annulus $\mathbb
T$ around the curve $S$ in $\mathbb{S}M$ which creates a two layered track as depicted in \cref{fig:fig1}. We will refer to $\mathbb{T}$ later as the surgery track. For each point $x\in S$, and a small enough local chart $U_x$ of $\mathbb{S}M$ around the point $x$, $\mathbb{T}\cap U_x$ is diffeomorphic to a rectangular strip. As a result of the construction of the set $\mathbb{T}\in \mathbb{S}M$, we have two square regions stacked on top of each other, $S_1$ on top and $S_2$ on bottom, along with two different lobes as shown in \cref{fig:fig1}, and the center of these two squares project under the canonical projection to the point $\beta(t_0)=\beta(t_1)$, and further we orient the surgery track $\mathbb{T}$ in such a way that at every point on $\mathbb{T}$ there is an $\mathbb{S}^{1}$ fiber transverse to $\mathbb{T}$ and passing through the point. For every point in one of the two layers of the double layer region of $\mathbb{T}$, the $\mathbb{S}^{1}$ fiber that passes transversely through the point, also passes $\mathbb{T}$ transversely through exactly one point on the other layer. The shearing squares $S_1,S_2$ are taken to be such that the same $\mathbb{S}^{1}$ fiber over the double layer only intersects $S_{1}$ once and $S_{2}$ once, such as shown in Figure 5. Also, any $\mathbb{S}^{1}$ fiber that intersects one of the lobes of $\mathbb{T}$, intersects it exactly once. 

\begin{definition}
Consider the set of $\mathbb{S}^{1}$ fibers that intersect the shearing track $\mathbb{T}$. The union $\widetilde{S}$ of all these fibers is the surgery region.
\end{definition}

Further, the vertical separation between the two square regions is $d$. Also, we will consider a parametrization for each of the fibers of the surgery region so that the origin on each such fiber is within a distance $d/2$ from the point(s) at which it intersects the the surgery track (either once or twice).

 We will show that the restriction of the altered flow after the surgery, to this surgery region, becomes ergodic, whereas in the complement $\mathbb{S}M\setminus \widetilde{S}$ we still get the trivial periodic flow.
 
We consider an initial trivial periodic flow $\mathbf{T}_t$ on $\mathbb{S}M$, which in local coordinates, on a given $\mathbb{S}^{1}$ fiber passing through the point $ x \in \mathbb{T}$, is simply a unit speed periodic flow $\mathbf{T}_{t}(x,\theta)=\theta+t (\text{mod} 1)$ on the fiber over $x$. Upon performing a modified Dehn surgery on $\mathbb{S}M$ by means of a shear map $\tilde{f}$ on $\mathbb{T}$, which we describe later, we will alter this simple periodic fiber flow; every time the flow encounters the surgery track $\mathbb{T}$ at a point $x\in \mathbb{T}$, the flow is taken to the periodic flow $\mathbf{T}_{t}$ on the fiber that intersects $\mathbb{T}$ at the point $f(x)\in \mathbb{T}$.

We define $\Tilde{f}$ first through the map $f$ on a subset $T$ of the torus below. We consider a shear profile which is linear across $T$, with slope $\alpha$.  

\begin{definition}
$f(x,y)=(x+\alpha (y-y_0),y)$ for $y_0\leq y\leq y_1$ with $\alpha (y_1 -y_0)=k$ for some positive integer $k$, defined on: $T= \{(x,y):0\leq x\leq 1, y_0 \leq y\leq y_1\}$ considered as a subspace of the torus $\mathbb{R}^{2}/\mathbb{Z}^{2}$. We term $k$ the winding number.
\end{definition}

We now twist $T$ into the ``two-layered" track $\mathbb{T}$ of \cref{fig:fig1}. The map is now linked over the double layer, by the periodic fiber flow over this surgered region, as explained above. See Figure 5. The width of the track $w=y_1 -y_0$ is taken to be so that $w\ll 1$.

We show that even though the initial flow is trivial and completely periodic on the unit tangent bundle, upon performing the modified Dehn surgery, the surgered flow is ergodic. 

Upon the modified Dehn surgery, the surgered flow can be described as follows: the flow $\mathbf{T}_{t}$ is taken to be counterclockwise on each fiber of the surgery region and this counterclockwise description is consistent for all the fibers of the surgery region. Whenever the flow encounters the surgery track `from below' at a point $x\in \mathbb{T}$ in a sense that is again consistent across the surgery track, the flow experiences a shear by the map $\Tilde{f}$ to reach $\tilde{f}(x)$ and then leaves the track `above' from the point $\tilde{f}(x)$ in a way that makes sense across the track $\mathbb{T}$. This is true for the double layer region as well. A typical part of the orbit that encounters both the double layers is shown in Figure 5.

As in the picture shown in \cref{fig:fig2}, one can equivalently describe the map $\Tilde{f}$ on the double layer as well as the lobes as a succession of two shears on the domain shown in Figure 3, one horizontally which we term $F$, and the other vertical shear which we term as $G$, so that the map becomes, in the domain $\{(x,y):x\in [0,1],y\in [y_0,y_1]\}$ for $F$ and the domain $\{(x,y): x\in [x_0,x_1],y\in [0,1] \}$ for $G$:

\begin{equation}\label{eq:eq2}
F\cdot
\begin{pmatrix}
x\\
y-y_0
\end{pmatrix}=\begin{pmatrix}
1 & \alpha \\
0 & 1 
\end{pmatrix} \cdot
\begin{pmatrix}
x\\
y-y_0
\end{pmatrix},
\end{equation}

\begin{equation}\label{eq:eq2}
G\cdot
\begin{pmatrix}
x-x_0\\
y
\end{pmatrix}=\begin{pmatrix}
1 & 0 \\
-\alpha & 1 
\end{pmatrix} \cdot
\begin{pmatrix}
x-x_0\\
y
\end{pmatrix},
\end{equation}

Further, the map $F$ is the identity in the region $[0,y_0]\cup [y_1,1]$ and the map $G$ is the identity in the region $[0,x_0]\cup [x_1,1]$. The maps $F$ and $G$ link together in the central square region with the vertices $ABCD$. 

The surgered flow restricted to the surgery region is called $\Psi$. With the linked twist map description, both the lower and upper squares coincide and we have the domain of Figure 3, and the orbit of any point $x$ in this domain under the linked twist map $H=G\cdot F$ is actually a subset of the orbit of $x$ under successive horizontal and vertical shears $F,G$. But ergodicity under the map $H$ also obviously gives ergodicity under the map which is a succession of the horizontal and vertical shears. The map $H$ shows hyperbolic behavior.

While in a proper Dehn surgery, one alters the manifold to recover another smooth manifold, as described in \cite{FH2013},\cite{FH2021}, our surgered manifold is not smooth; we only enforce a $C^{0}$ joining of the shear to the boundary of $\mathbb{T}$. We only apply the machinery of uniform hyperbolicity to achieve an ergodic flow in such a surgered manifold.

We refer the reader to the recent manuscripts  \cite{FH2013},\cite{FH2021} for more background for this work. The existence of a Smale horseshoe for the surgered flow is shown in \cite{HH2023}. In the discussion subsequent to the Main theorem in \cite{HH2023}, the authors describe the work on showing the volume restricted to the orbit saturation of the surgery annulus is hyperbolic and ergodic, which is what we show in this manuscript.

We also refer the reader to earlier works of \cite{BE, Woj} which establish ergodicity in the linked twist mapping when the twists reinforce each other, and also \cite{Devaney} which establishes the presence of a horseshoe in linked twist mappings.

For the situation of a Dehn surgery with $C^{k}$ boundary shear profile on $\mathbb{T}$, the problem of studying even some basic properties of the corresponding projected map on the shearing track becomes difficult, which would be the object of future analysis.

When we unravel the surgery track, we get the schematic picture in \cref{fig:second} with the left edge of the track in \cref{fig:second} identified with the right edge of the track . The dynamics is under the map $\tilde{f}$, which is linked between the squares $S_1$ and $S_2$, but otherwise is actually just equivalent to the original twist map of Definition 2.

%To show ergodicity of the map and all its powers, by standard (references, Katok Strelcyn, Przytycki) Pesin theory, we only need to show:

When the shear is made smooth and thus weak enough at the boundary of the shearing region such as in Figure 4(a), with the boundary identification we have, the problem of determining the orbit structure appears to become difficult, unlike in the case considered in \cite{BE} where in the linked twist map the shears in the central region reinforce each other and thus we escape into the bulk of the square $S$ where we again experience reinforcing strong shears. In our opposing identification, orbits can spend a long time in the corners with successively weak shears, and the study of the orbit structure within the central region $S_1 \cup S_2$ becomes difficult.

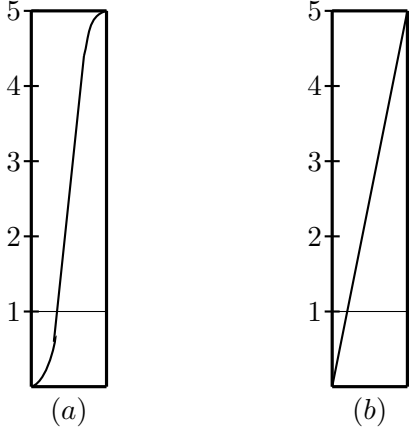
\begin{figure}
\centering
    \begin{tikzpicture}
        \draw[very thick] (0,0) to (1,0);
        \draw[very thick] (1,0) to (1,5);
        \draw[very thick] (0,0) to (0,5);
        \draw[very thick] (0,5) to (1,5);
        \draw[thick] (0.1,1) to (-0.1,1);
        \draw[thick] (0.1,2) to (-0.1,2);
        \draw[thick] (0.1,3) to (-0.1,3);
        \draw[thick] (0.1,4) to (-0.1,4);
        \draw[thick] (0.1,5) to (-0.1,5);    
        \draw[thick] (0,0) to (1,5);
        \draw (0,1) to (1,1);
        \node [left] at (0,1) {$1$};
        \node [left] at (0,2) {$2$};
        \node [left] at (0,3) {$3$};
        \node [left] at (0,4) {$4$};
        \node [left] at (0,5) {$5$};
        \node [below] at (0.5,0) {$(b)$};
        
        \draw[very thick] (-4,0) to (-3,0);
        \draw[very thick] (-3,0) to (-3,5);
        \draw[very thick] (-4,0) to (-4,5);
        \draw[very thick] (-4,5) to (-3,5);
        \draw[thick] (-3.9,1) to (-4.1,1);
        \draw[thick] (-3.9,2) to (-4.1,2);
        \draw[thick] (-3.9,3) to (-4.1,3);
        \draw[thick] (-3.9,4) to (-4.1,4);
        \draw[thick] (-3.9,5) to (-4.1,5);
        \draw (-4,1) to (-3,1);
        \node [left] at (-4,1) {$1$};
        \node [left] at (-4,2) {$2$};
        \node [left] at (-4,3) {$3$};
        \node [left] at (-4,4) {$4$};
        \node [left] at (-4,5) {$5$};
        \node [below] at (-3.5,0) {$(a)$};
        \draw[thick] (-4,0) to [out=15, in=75](-3.7,0.6);
        \draw[thick] (-3.3,4.4) to [out=75, in=195](-3,5);
        \draw[thick] (-3.7,0.6) to (-3.3,4.4);
    \end{tikzpicture} 
    \caption{A uniform shear profile $f$ shown in part(b), whose derivative is discontinuous at the boundary, with $k=5$. A nonuniform shear profile $f$ that is $C^{m}$ (or could be made $C^{\infty}$) at the two boundaries, where also $k=5$, is shown in part (a).}
    \label{fig:third}
\end{figure}

We prove the following theorem in Section 2.

\begin{theorem}
    The map $\tilde{f}$ with the winding number $k\geq 2$ on $\mathbb{T}$ is ergodic with respect to the Lebesgue measure, when the shear parameter $\alpha>6.23$. In fact this map is Bernoulli.\footnote{which is more that what we need.}
\end{theorem}

As a result, we establish the main result for the surgered flow in this paper:

\begin{corollary}
    The surgered flow $\Psi$ on $\mathbb{S}M$, restricted to the surgery region, is ergodic with respect to the volume measure.
\end{corollary}

\begin{proof}
    Consider any subset $A$ of the surgery region, that is invariant under the surgered flow $\Psi$. The projection of $A$ to the annulus $\mathbb{T}$ is then invariant under the shear map $\tilde{f}$ and thus has Lebesgue measure zero on $\mathbb{T}$ since the map $\tilde{f}$ is ergodic by Theorem 1, and thus the set $A$ itself also has zero volume measure in $\mathbb{S}M$.
\end{proof}

Further, even though the map $f$ on $\mathbb{T}$ has the Bernoulli hence mixing property, we are also able to establish that the surgered flow is not weakly mixing.

\begin{theorem}
    The surgered flow $\psi_t$ is not weakly-mixing.
\end{theorem}

\begin{proof}
Consider the width $d_1$ between the two squares in the surgery region, as shown in the figure. Consider an arbitrary subset of the surgery region of $\mathbb{S}M$  with the following property: the set $A$ projects to a rectangle $R$ (and thus of positive Lebesgue measure on $\mathbb{T}$) lying on any one of the lobes of track $T$, and in local charts on $\mathbb{S}M$ where the coordinates are given by $(t,\theta), t\in [a,b]\times[c,e], \theta\in [0,1] ])$, with $R\subset [a,b]\times [c,e]$, we have $A=R\times [-\eps,\eps]$ with $2\eps <d_1$. In other words, the set is a cube with uniform width $\epsilon$ in the direction of the fibers $\mathbb{S}^{1}$ in a local chart. Consider any other set $B$ that has the same property, also with a width $2\epsilon$. Now consider the flow of the set $A$ under $\Psi$. Since $\epsilon<d_1$,the set $\Psi_t(A)$ can only be decomposed as disjoint unions $\sqcup_{i=1}^{N(t)}A_i(t)$ with each $A_i \subset \mathbb{S}M$, $N(t)\to \infty$ as $t\to \infty$, but where each of the sets $A_i(t)$ continue to have width $2\epsilon<d_1$ in the direction of the fibers.
    
The set $\Psi_t(A)$ as $t\to \infty$ spends a positive fraction in each time interval of length $2\pi$, uniformly a distance $O(\epsilon)$ away from the track $\mathbb{T}$, and thus for a positive fraction of the time, the set $\Psi_t(A)$ also has null intersection with $B$, and the flow is not weakly mixing. 
\end{proof}

%To show that the surgered flow is ergodic on the unit tangent bundle, it suffices to only show that the linked shear map on $\mathbb{T}$, which is the discrete map under study, is itself ergodic with respect to the surface measure on the surgery track. If the surgered flow was to have any invariant set under the complete flow, then such an invariant set must intersect the surgery region, and it is easy to see that it would force us (need to expand on this, should not be hard) to find an invariant set in the surgery track $\mathbb{T}$ as well, contradicting the ergodicity of the discrete map that we will have shown. Locally, away from the square region $S_1\cup S_2$, this flow is locally like a suspension flow. Even if the map $\mathbb{T}$ on the surgery region is mixing, the surgered flow is no longer mixing since this is locally like a suspension flow, in the region of the lobes in \cref{fig:fig1}. This needs to bewritten more properly, but I am hoping this correct; in fact the flow would also not be weakly mixing.

%\rule{10cm}{0.11mm}
\bigskip
From now on it is enough to work in the domain $\tilde{T}:=W\cup V$ of Figure 3 and the map $H$. \footnote{We note that because of our boundary identification, we might end up having a segment in the square $S$ that suffers one horizontal shear and part of this horizontally sheared segment crosses the line $\overline{PQ}$ and enters the top lobe vertically.This does not happen with the usual boundary identifications of \cite{Prz,BE}.}

Following the arguments of \cite{Prz} it suffices to show that $H^{m}(\gamma^{u}(x))$ eventually contains either a vertical or a horizontal segment through the square $S$ for some large enough $m$. Since $k\geq 2$, once we have a vertical segment through $S$ belonging to some $H^{m}(\gamma^{u}(x))$, then because $k\geq 2$ we also have a horizontal segment through $S$ belonging to $F\cdot H^{m}(\gamma^{u}(x))$ and then under one further shear we have another vertical segment belonging to $H^{m+1}(\gamma^{u}(x))$ and then also for all greater integers. The same argument would show that $H^{-n}(\gamma^{s}(y))$ also has either a horizontal or vertical segment for all large enough positive integers $n$. This shows the Bernoulli property for the map, given Theorem 4 below.

Note that the composite shear map $H$ (ref \cite{Prz}) is given, for $\alpha>2$,

\begin{equation}\label{eq:eq1}
H=G\cdot F=\begin{pmatrix}
1 & 0 \\
-\alpha & 1 
\end{pmatrix} \cdot
\begin{pmatrix}
1 & \alpha \\
0 & 1 
\end{pmatrix},
\end{equation}

whereas, the inverse composite shear map is given by 

\begin{equation}\label{eq:eq2}
H^{-1}=F^{-1}\cdot G^{-1}=\begin{pmatrix}
1 & -\alpha \\
0 & 1 
\end{pmatrix} \cdot
\begin{pmatrix}
1 & 0 \\
\alpha & 1 
\end{pmatrix}.
\end{equation}

The eigenvalues of the map $DH$ are given by:

\begin{equation}\label{eq:eq3}
\lambda_{\pm}=\frac{-\alpha^{2}+2\pm \sqrt{\alpha^{4}-4\alpha^{2}}}{2}.
\end{equation}

This shows we have hyperbolicity when $\alpha\geq 2$.

\begin{figure}[h]
\centering
\includegraphics[width=1.1\textwidth]{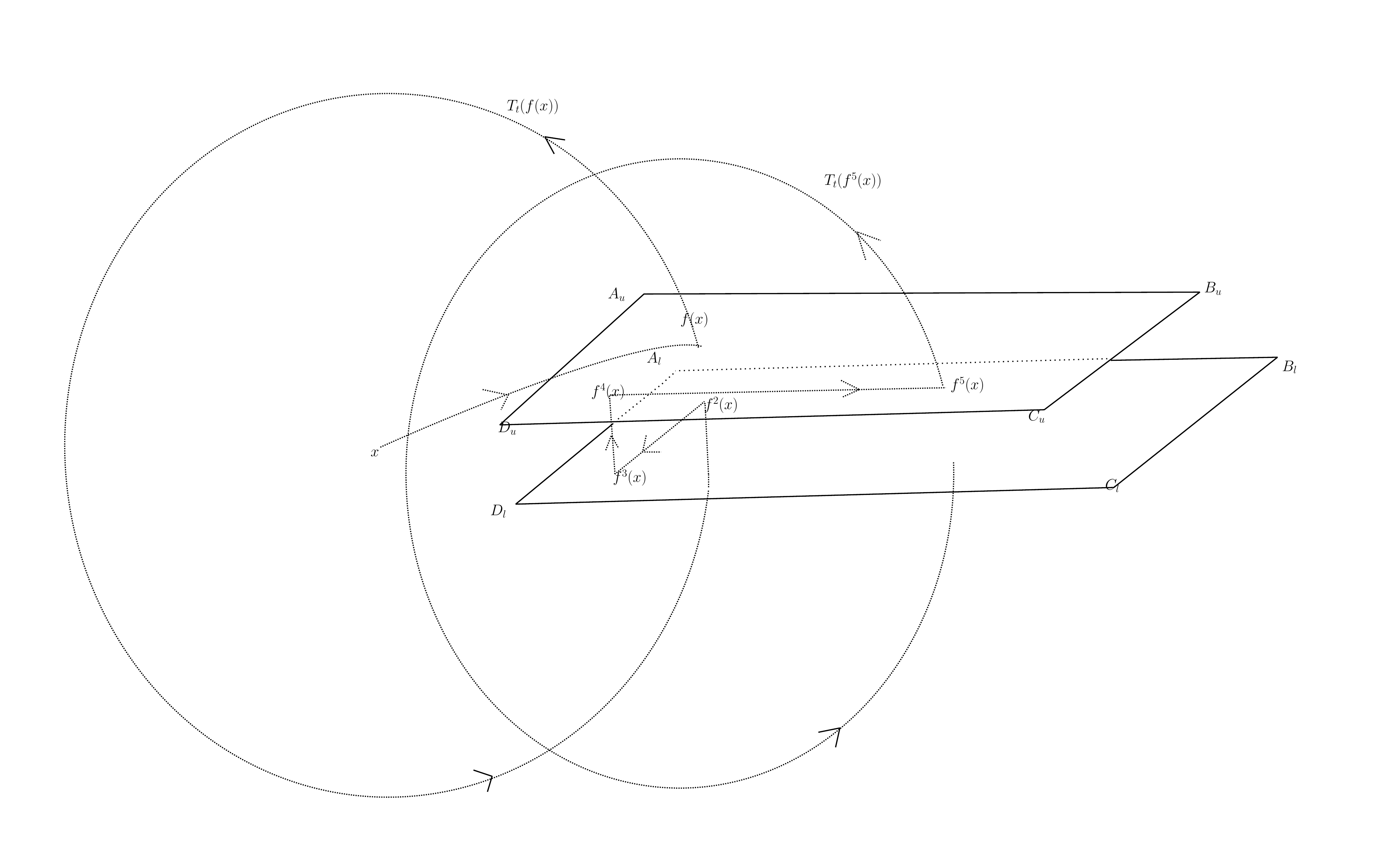}
%\huge\input{drawing.pdf_tex}
\caption{The point $x$ moves to $f(x)$ on the top square, and then the ensuing dynamics is shown on the two layered square. The perpendicular distance between the two square layers is given by $d_1$.}
\label{fig:inksone3}
\end{figure}

It is easily verified that the expanding eigenvector $(\psi_1,\psi_2)$ and the contracting eigenvector $(\eta_1,\eta_2)$ satisfy the following relation, 

\begin{equation}\label{eq:eq4}
\frac{\psi_1}{\psi_2}=-\Big(\frac{\alpha}{2}\Big)+\sqrt{\big(\frac{\alpha}{2}\big)^{2}-1} =\frac{\eta_2}{\eta_1}=L
\end{equation}

These above, and the expressions in \cref{eq:eq1} and \cref{eq:eq2} show the symmetry in $\gamma^{s})(x), \gamma^{u}(x)$ under the forward and inverse maps respectively. It will be enough to focus only on the dynamics of the unstable manifold $\gamma^{u}(x)$ under the forward map.

We note that in related work on the linked twist map \cite{Pat}, under the more usual and well studied boundary identifications such as in \cite{Prz,BE} for the linked twist map, we show that one can get ergodicity with $k=1$ but that then the Bernoulli property is lost. The optimal shear parameter for which the ergodicity was achieved in \cite{Prz} is also reduced in \cite{Pat} from the original value of $\alpha=4.15$ to $\alpha=3.47$ in the most general case without any assumptions on the dimensions of the track, but again with the usual boundary identification. One should be able to modify and adopt those arguments to the setting of the boundary identification of Figure 3 here.

Similar to the cases in \cite{BE,Prz}, if not for a rational orbit, the set of points of which has measure zero, every point of $S$ returns again to $S$. By a standard argument, (See Prop 4.4 of \cite{BE} and also see \cite{Prz}) these points return under $H$ to $S$ with positive frequency, and thus Lyapunov exponents exist and local stable and unstable manifolds exist. Then to show ergodicity of $H$ and all its powers, we can use the following theorem as also found in \cite{Prz}.

\begin{theorem}\label{thm:thm1}

For Lebesgue almost every $x,y\in S_1$, denote the local stable and unstable manifolds at $x$ and $y$ respectively by $\gamma^{s}(x) \ \text{and} \ \gamma^{u}(y)$, if $H^{m}(\gamma^{u}(x))$ intersects $H^{-n}(\gamma^{s}(y))$ for all positive integers $m,n$ large enough, then $H$ and all its powers are ergodic. 

\end{theorem}

\section{Proof of Theorem 1.}
\subsection{The dynamics of the linked shear map.}

Given that we only consider linear shears, as stated earlier for any point $x$, the stable and unstable manifolds $\gamma^{s}(x),\gamma^{u}(x)$ are linear segments whose slopes are respectively $L$ and $1/L$ from \cref{eq:eq4}. The unstable manifold is a linear segment lying on the left boundary of the cone $C$ given by $\{(v,w):L\leq v/w\leq 0\}$, because of \cref{eq:eq4}. 

Consider points in $S$. Under successive shears in the lobes, eventually the image under $H$ of this segment either enters the square $S$ vertically in \cref{fig:fig2}, again within a cone $C$ upon entry, or enters $S$ from the left in \cref{fig:fig1} horizontally within the cone $C'$, which is a rotation by $\pi/2$ of $C$. 

Further, we assume that the dimensions of the central square regions are small compared to the total length of the unfolded track (which is taken to be of unit length). In particular, in the ensuing analysis when we talk about points just to the left or right of the edges of the square region, they will be understood without any difficulty, since the lobes are considered large enough so that no ambiguity arises about these notions of `left' and 'right' near the boundary of the square regions.

Under exactly one shear, the boundary of $C$ given by the line $v/w=L$ maps exactly to the corresponding boundary of $C'$, whereas after a sufficiently large number of shears in the lobes, the images of the unstable manifolds are within the interior of the cone $C$ or $C'$. This follows from \cref{eq:eq4}, which is equivalent to $L=- 1/(\alpha+L)$ which shows that under one iteration of a horizontal shear , the boundary of $C$ is mapped exactly to the boundary of $C'$, and under further horizontal(vertical) shears, a segment in $C'(C)$ gets mapped within $C'(C)$, and upon further iterations, a segment in the interior of $C(C')$ under a horizontal(vertical) shear gets mapped into the interior of the cone $C'(C)$.

\begin{figure}[h]
\centering
    \begin{tikzpicture}[scale=0.80]
        \draw[very thick] (-7,0) to (-9,3);
        \draw[thick] (-7,0) to (-7,3);
        \node [below] at (-8,3) {Cone $C$};
        \draw[very thick] (-4,0) to (3,3);
        \draw[thick] (-4,0) to (3,0);
        \node [right] at (2,1) {Cone $C'$};
        \draw[->] (-7.6,1) to [out=45, in=135](-0.35,1.6);
        
    \end{tikzpicture} 
    \caption{Passing from the vertical cone $C$ to the cone $C'$, through the horizontal map $F$ in the square region $S$. The bold edge of the cone $C$ is mapped to the bold edge of the cone $C'$ as shown in the above figure. Similarly, when one passes from the cone $C'$ to the cone $C$ through the vertical map $G$, the bold edge of $C'$ essentially gets mapped to the bold edge of $C$.}
    \label{fig:fourth}
\end{figure}
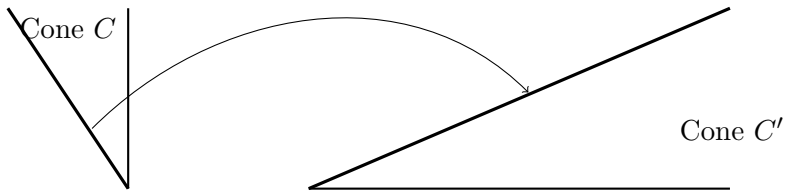

We ensure the shear parameter is large enough that eventually for all large enough iterations $n$ of the map, we have a v- segment through $S$, contained in $H^{n}(\gamma^{u}(x))$. Following the notation of \cite{Prz}, by an h-segment we mean a straight line segment that intersects both the right and left sides of $S$. A v-segment is a straight line segment that intersects both the top and bottom sides of $S$. For sake of simplicity of notation, consider the first return map $H_s$ to $S$. 

Start with any segment $\gamma\subset H_{s}^{p}(\gamma^{u}(x))$ in $S$, within the cone $C$, with vertical length $l_v(\gamma)$, for any positive integer $p$. Whenever a segment enters $S$ vertically under the map $H_s$, it enters within the cone $C$, and when a segment enters $S$ horizontally it does so within the cone $C'$, and we will ensure that eventually for some integers $\tilde{n_0},n_0>0$ a segment $L_1\subset H^{\tilde{n_0}}(\gamma) \subset H^{n_0}(\gamma^{u}(x)) $ inserted into $S$ has vertical length $l_v(L_1)>\delta l_v(\gamma)$, i.e. greater than the length of the original segment by a multiplicative factor of $\delta>1$ (here $\delta$ is uniform constant, independent of the segments $\gamma$), or eventually for some integer $m_0$ a segment $L_2\subset F\cdot H^{m_0}(\gamma^{u})(x)$ inserted into $S$ which has horizontal length $l_v(L_1)>\delta l_v(\gamma)$, i.e. greater than the length of the original segment by a multiplicative factor of $\delta>1$.

%Further, when a segment enters $S$ horizontally within the cone $C'$, under the map $H_r$, we also ensure that the horizontal length of this segment $L_k$, $k\geq 2$, that enters the top square has length $l_h(L_k)$ which is greater by the same factor $\delta$, compared to the corresponding length of the prior segment $L_{k-1}$ that entered in $S_1\cup S_2$ under $H_r$, which if it was a `nearly vertical' (`nearly horizontal') segment then we consider the vertical (horizontal) length.

%When a segment within a cone $C$ enters the bottom square, it is transported to the top square by the periodic flow, and similarly the segment entering the top square, within the cone $C'$, is taken into the bottom square by this periodic flow. In either case, these segments are sheared again, and we have successively larger 'nearly vertical' segments entering the bottom square from above or successively larger 'nearly horizontal' segments entering the top square from the left, such that eventually we have a horizontal h-segment on the top as well as a vertical v- segment on the bottom. 

%Once we have a vertical segment in the bottom square, the periodic flow brings it to the top layer, and the next shear is strong enough that it also creates a horizontal segment on the top layer.  

Because of the exponential growth of the length of the segments above, eventually we will get a large enough segment through $S$ that is either a vertical v-segment or a horizontal h-segment. In the next iteration of the map, as noted earlier, because $k\geq 2$, we achieve both vertical and horizontal segments in $S$.

Note that when we get both v-segments and h-segments in $S$ under the map $H$ in the domain $\tilde{T}$ of Figure 3, we correspondingly get both v-segments and h-segments in both the top and bottom layers $S_1,S_2$ of $\mathbb{T}$ under the map $\tilde{f}$. 

\begin{figure}\label{fig7}
\centering
    \begin{tikzpicture}
        \draw[very thick] (-7,1.5) to (-7,-1.5);
        \draw[very thick] (6,1.5) to (6,-1.5);
        \draw[very thick] (8,1.5) to (8,-1.5);
        \draw[very thick] (8,1.5) to (-7,1.5);
        \draw[very thick] (8,-1.5) to (-7,-1.5);
        \draw[very thick] (4,-1.5) to (4,1.5);
        \draw[very thick] (-3,-1.5) to (-3,1.5);
        \draw[very thick] (-1,-1.5) to (-1,1.5);
        \node [above] at (-2,-1) {$S_1$};
        \node [above] at (5,-1) {$S_2$};
        \node [below] at (2.67,0) {$p_1$};
        \node [below] at (-4.43,0) {$p_2$};
        \draw [thick] (-4.47,-0.1) to (-4.40,0.1);
        \draw [thick] (-4.47,0.1) to (-4.40,-0.1);
        \draw [very thick] (-2.64,0.75) to (-2.14,-0.4);
        \draw [very thick] (1.7,-0.4) to (4.5,0.75);
        \draw [fill] (-5.3,0) circle [radius=0.05];
        \draw [fill] (-5.8,0) circle [radius=0.05];
        \draw [fill] (-6.3,0) circle [radius=0.05];
        \draw [fill] (-6.8,0) circle [radius=0.05];
        \draw [fill] (-4.3,0) circle [radius=0.05];
        \draw [fill] (-4.8,0) circle [radius=0.05];
        \draw [fill] (-3.8,0) circle [radius=0.05];
        \draw [fill] (-3.3,0) circle [radius=0.05];
        \draw [fill] (-2.8,0) circle [radius=0.05];
        \draw [fill] (-2.3,0) circle [radius=0.05];
        \draw [fill] (-1.8,0) circle [radius=0.05];
        \draw [fill] (-1.3,0) circle [radius=0.05];
        \draw [fill] (-0.8,0) circle [radius=0.05];
        \draw [fill] (-0.3,0) circle [radius=0.05];
        \draw [fill] (0.2,0) circle [radius=0.05];
        \draw [fill] (0.7,0) circle [radius=0.05];
        \draw [fill] (1.2,0) circle [radius=0.05];
        \draw [fill] (1.7,0) circle [radius=0.05];
        \draw [fill] (2.2,0) circle [radius=0.05];
        \draw [thick] (2.7,-0.1) to (2.7,0.1);
        \draw [thick] (2.7,0.1) to (2.63,-0.1);
        \draw [fill] (3.2,0) circle [radius=0.05];
        \draw [fill] (3.7,0) circle [radius=0.05];
        \draw [fill] (4.2,0) circle [radius=0.05];
        \draw [fill] (4.7,0) circle [radius=0.05];
        \draw [fill] (5.2,0) circle [radius=0.05];
        \draw [fill] (5.7,0) circle [radius=0.05];
        \draw [fill] (6.2,0) circle [radius=0.05];
        \draw [fill] (6.7,0) circle [radius=0.05];
        \draw [fill] (7.2,0) circle [radius=0.05];
        \draw [fill] (7.7,0) circle [radius=0.05];
        \draw [<->] (2.67, -0.8) to (4.0,-0.8);
        \node [below] at (3.3,-0.8) {$D$};
        \draw [dashed] (2.67,-0.8) to (2.67,0);
        \node [left] at (-2.9,1.2) {$LE_1$};
        \node [right] at (-1,1.2) {$RE_1$};
        \node [left] at (6,1.2) {$RE_2$};
        \node [left] at (4,1.2) {$LE_2$};
        \node [below] at (-1.3,0) {$R_1$};
        \node [below] at (-0.7,0) {$T_1$};
        \node [below] at (5.7,0) {$R_2$};
        \node [below] at (6.3,0) {$T_2$};

    \end{tikzpicture} 
    \caption{Unfolding the linked twist map. The two square regions are labelled as $S_1$ and $S_2$. The distance between the right edge $RE_1$ of $S_1$ and the left edge $LE_2$ of $S_2$ can without loss of generality be taken to be equal to the distance between the right edge $RE_2$ of $S_2$ and the left edge $LE_1$ of $S_1$ under the identification, i.e. we consider the lobes to be symmetric. The initial segment $\gamma$ is in the square $S_1$ and $H_{s}(\gamma)$ intersects the left edge of the square $S_2$.}
    \label{fig:second}
\end{figure}
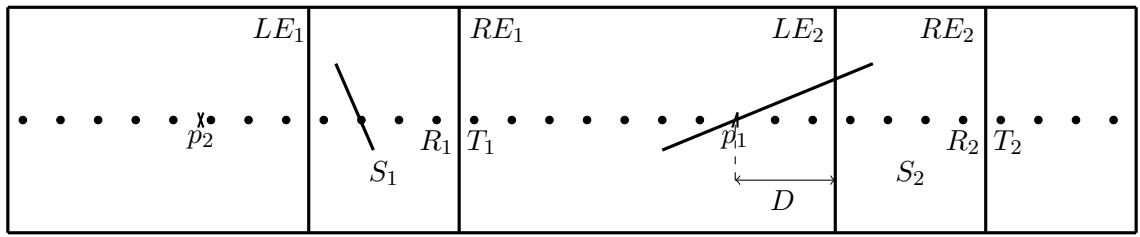

\begin{figure}
\centering
    \begin{tikzpicture}
        \draw[very thick] (-7,2) to (-7,-2);
        \draw[very thick] (-1,2) to (-7,2);
        \draw[very thick] (-1,-2) to (-7,-2);
        \draw[very thick] (-3,-2) to (-3,2);
        \draw[very thick] (-1,-2) to (-1,2);
        \node [above] at (-2,-1) {$S_2$};
        \node [above] at (-4.2,0) {$p_1$};
        \draw [thick] (-6.1,-0.9) to (-1.3,1.4);
        \draw [fill] (-6.7,0) circle [radius=0.05];
        \draw [fill] (-6.2,0) circle [radius=0.05];
        \draw [fill] (-5.7,0) circle [radius=0.05];
        \draw [fill] (-5.2,0) circle [radius=0.05];
        \draw [fill] (-4.2,0) circle [radius=0.05];
        \draw [fill] (-4.7,0) circle [radius=0.05];
        \draw [fill] (-3.7,0) circle [radius=0.05];
        \draw [fill] (-3.2,0) circle [radius=0.05];
        \draw [fill] (-2.7,0) circle [radius=0.05];
        \draw [fill] (-2.2,0) circle [radius=0.05];
        \draw [fill] (-1.7,0) circle [radius=0.05];
        \draw [fill] (-1.2,0) circle [radius=0.05];
        \node [left] at (-2.9,1.65) {$LE_2$};
        \node [right] at (-1,1.65) {$RE_2$};
        \node [left] at (-5.3,-0.5) {$I_{1}$};
        \draw[fill] (-5.0,-0.35) circle [radius=0.03];
        \draw[fill] (-3.5,0.35) circle [radius=0.03];
        \node [right] at (-4.7,-0.3) {$I'_{2}$};
        \node [left] at (-3.5,0.35) {$I''_{2}$};
        \node [above] at (-3.2,0.45) {$I_{3}$};
        \node [above] at (-2.2,0.95) {$I_{4}$};

    \end{tikzpicture} \label{fig8}
    \caption{An expanded picture of the case of first return to the square $S_2$ as depicted in \cref{fig:second}, and identifications of the segments $I_1,I_2',I_2'',I_3,I_4$ used in the analysis. We have that the disjoint union of the segments $I'_2 \cup I''_{2}=I_{2}.$}
    \label{fig:third}
\end{figure}
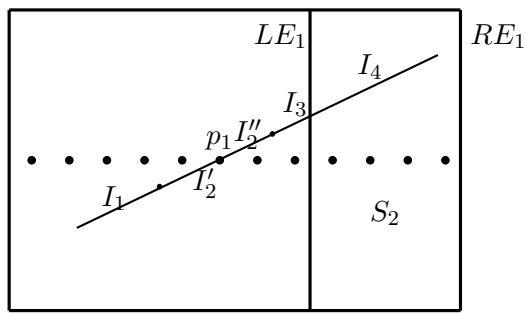

Consider without loss of generality that the segment $\gamma$ lies within the cone $C$ within $S_1$ as in \cref{fig7}. It will be enough to separately consider the two subcases in Sections 2.1.1 and 2.1.2 .

\subsubsection{First return of $\gamma$ intersects only one square}

First, we are interested in the dynamics where, the first return of $\gamma$ under the iterations under the $H$ map intersects just one of the two squares $S_1, S_2$. There are two possibilities here: either it intersects one of the left edges $LE_1$ or $LE_2$ or that it intersects one of the right edges $RE_1$ or $RE_2$. It will be enough to consider the case shown in Figures 7 and 8, where it intersects one of the left edge of one of the two squares $S_1,S_2$, and the analysis in the other case is identical.

For the sake of argument, from now on we consider the `unfolded' linked twist map shown in \cref{fig:second}, where the segment under first return intersects the left edge of the square $S_2$. The segment $\gamma$ itself lies in the square $S_1$. This segment now belongs to some $H^{m_1}(\gamma)$.

%The argument would run identically, as can be seen, if the first return to either of the two squares intersected the right edge $RE_1$ or $RE_2$. The case where in the first return the segment intersects both the squares, is treated separately afterwards.

%There is thus a positive integer $m_1\geq 1$ such that $H^{m_1}(\gamma) \cap (S_{1}\cup S_{2}) )\neq \phi$ and $H^{m}(\gamma) \cap (S_{i}) ) =\phi, \  \forall m_1>m>0$ where $i$ is either 1 or 2.

Following \cite{Prz}, as in \cite{Pat}, we define the distance $d$ as follows: we are looking for the unique integer $q$ such that $1/q < \alpha l_v (I_2)$ and $1/(q-1) \geq \alpha l_v (I_2)$, which is $q= \lfloor \frac{1}{\alpha l_v (I_2)} +1 \rfloor$. Given the segment $I_2$, let the vertical endpoints of $I_2$ be $y_1,y_2=y_1+l_{v}(I_2)$. In this case, under a forward iterate of the map, one endpoint moves forward by a distance of $\alpha y_1$ and the other endpoint moves forward by a distance $\alpha(y_1 +l_v(I_2))$. We seek a point on $I_2$ such that this point moves under this iteration of the map by a rational amount of $p/q$ with $\alpha y_1 \leq p/q < \alpha(y_1 +l_v(I_2))$, with $q=\lfloor \frac{1}{\alpha l_v (I_2)} +1 \rfloor$ as above. Clearly such a point would exist. The period of such a orbit, depending on whether we find $p,q$ coprime, is some divisor of $q$ and the distance between succesive points of the orbits is some multiple of $d=\frac{1}{\lfloor \frac{1}{\alpha l_v (I_2)} +1 \rfloor}$ and thus at least this value. We could have chosen a larger value of $q$ to make the distance $d$ smaller, but this would need stronger shear than that needed with a value of $q$ that is the smallest possible.

Consider the situation in \cref{fig:second}. We divide the segment $H^{m_1}(\gamma)$ into four distinct parts $I_1,I_2,I_3,I_4$ as shown in \cref{fig8}, and we show that it is enough that we require the following conditions stated later.

We claim it is enough in this case if we have the requirement that there exists a constant $\delta>1$ such that all four of the following hold.

\begin{align}\label{eq:eq1}
    d\geq 2\delta l_{v}(\gamma) \\
    \alpha l_{v}(I_{2}'' \cup I_{3}) \geq \delta l_{v}(\gamma). \\
    \alpha l_{v}(I_1 \cup I'_2) +l_{h}(I_1\cup I'_2) \geq 3\delta l_v(\gamma) \\
    l_{h}(I''_2\cup I_3) \geq \delta l_v(\gamma) 
\end{align}
 
%\footnote{The case $(3)$ above is necessary because when the "tail end" is cut off by one of the two square regions, we would require that at least the one part that is cut off within the square itself be greater than $\delta l_v(\gamma)$, or the lower part of the segment carried outside the square itself has the length greater than $\delta l_v(\gamma)$.}

We outline the arguments below. As mentioned earlier, we assume that the segment $\gamma$ under consideration is inside the square $S_1$. We consider the specific point $p_1$ on the segment $I_2$ that has a rational orbit with $d$ being the distance between the nearest points of the orbit.

The case where this first return is to the square $S_1$ is entirely analogous to the analysis here. 

The cases of the second return happening to $S_1$ and the second return happening to $S_2$ have to be treated slightly differently, and that will be apparent in the argument below. 

The distance from $p_1$ to $LE_2$ is denoted $D$ as shown in \cref{fig:second}.

Consider the points of the rational orbit in the set $T=\{x: x \notin S_1 \cup S_2, d(x,LE_2)\in (\text{[max}(0,D-d/2) ,D+d/2)\  \text{or} \ d(x,LE_1)\in (\text{[max}(0,D-d/2) ,D+d/2)  \}$. By construction there is at most one point of the rational orbit in the distance range $[\max(0, D-d/2),D+d/2)$ from $LE_1$, and exactly the point $p_1$ of the rational orbit in the distance range $[\max(0, D-d/2,D+d/2))$ from $LE_2$.

We are interested in the first time the orbit of the point $p_1$ under the map $H$ returns to within a distance at most $D$ to the left of $LE_1$ or to the left of $LE_2$, or within any of the two squares itself. We call this point the point of second return. Consider the symmetrically placed point $p_2$ at a distance exactly $D$ from the left edge $LE_2$ of $S_2$, which in general is not part of the rational orbit.

In case the first return is so that $l_h(I_4)$, then we are done immediately.
\bigskip

Case(i): Suppose that this point of second return is to a point $Q$ such that $d(Q,LE_1)<(D-d/2)$ to the left of $LE_1$, in case $D>d/2$, and if such a point exists. In this case, there is always at least a horizontal length $D=l_h(I_2^{''}\cup I_3)$ (in fact a bigger length if at least one further iteration has taken place in between) that has not been cut off prior to returning to the distance at most $D$ from either of $LE_1$ or $LE_2$. $Q$  is at least a distance $d/2$ away from the point $p_2$, and $d/2<l_h(I_2^{''}\cup I_3)=D$ by construction in this case, thus a horizontal length at least $d/2$ has been pushed inside the square $S_2$.

\bigskip 

Case(ii): Now suppose instead that at the second return, the point of the orbit $H{^k}(p_1)$ for some $k\geq 1$ lies in the distance range $(\max(0,D-d/2),D)$ from the left edge $LE_1$ of $S_1$, or equivalently, at a distance less than or equal to $d/2$ to the right of $p_2$, and call this point $Q^{'}$. In this case, the successive horizontal lengths outside $S_1 \cup S_2$, as the orbit moves from $p_1$ to $Q'$, with at least one point in between $p_1$ and $Q'$, is always at least $\text{min}\big(d/2 + l_h(I_2^{''}\cup I_3), l_h(I_2^{''}\cup I_3)+\alpha l_v(I''_{2}\cup I_3)\big).$ This is because, the nearest point of the orbit to the left of $Q'$ exactly a distance $d$ from $Q'$, is at least a distance $d/2$ away from $p_2$ and by hypothesis that point belongs to the complement of $S_1 \cup S_2$. In the case we move directly from $p_1$ to $H(p_1)=Q'$ we will have at least a length $\alpha l_v(I_2^{''}\cup I_3)$ that is inserted inside the central square region $S_1$, or an h-segment in which case we are done, otherwise if there are further iterations in between, which means $H^{k}(p_1)=Q'$ for some $k\geq 2$, then from the above argument, we would still have a length of $d/2$ inserted inside the square $S_{1}$.\footnote{Note that if $k\geq 2$ in this case, then the successive `vertical lengths' $l_v(I_{k})$, say, that get cut off outside $S_1 \cup S_2$ become successive smaller, and we cannot control uniformly over $k$ the lengths of horizontal stretches $\alpha l_{v}(I_{k})$.}

\bigskip

Case(iii): Note first that the point of second return cannot be exactly the point $p_1$, since then the $H$ orbit of $p_1$ is disjoint from $S_1\cup S_2$ which is  contradiction since $\gamma\in S_1$.  In case the second return is to the region between $p_1$ and $LE_2$, the point of return is at least at a distance $d$ to the right of $p_1$, and since till that point we always have a horizontal length $D=l_{h}(I''_2 \cup I_3)$ outside of $S_1 \cup S_2$, with $D>d$ by construction in this case, a horizontal length at least $d$ would be inserted inside the square $S_2$.

\bigskip

Case(iv): Consider the case when the second return of $p_1$ is to a point inside $S_1\cup S_2$, except for the points $R_1$ or $R_2$. The amount of horizontal length within the square is at least $\text{min}(d,l_{h}(I''_2\cup I_3))+\alpha l_v(I''_{2}\cup I_3)$. The term $d$ appears since $H^{k}(p_1)$ might be the point just to the left of $R_1$ or $R_2$ and either of $R_1,R_2$ may be arbitrarily close to the right edges $RE_1$ or $RE_2$. In either case, we would be done.

\bigskip

 Case(v): If the second return is to $R_1$ (or equivalently $R_2$) in that case again, we would either have at least an amount $l_{h}(I''_2\cup I_3))$ to the right of the point $R_1$, in which case we are done because of Equation 8, or the segment is cut off by the right edge $RE_1$. This is because starting from $p_1$, the orbit $H^{k}(p_1)$ for $k\geq 1$ may pass through a point a distance $D+\epsilon$ to the left of $LE_1$ for arbitrarily small $\epsilon$, and then reach the point $R_1$ and thus only possibly a horizontal distance of $l_h(I''_2\cup I_3))$ of the segment to the right of $R_1$ being inserted into the square $S_1$.

Consider in particular the case where we have segments that are intersecting either of $RE_1$ or $RE_2$. Because of the equation $(9)$ above, it must happen that when a segment gets cut off by the right edge of $S_1$, \footnote{the case where it first gets cut off by $S_{2}$ first is also entirely analogous} either a horizontal length greater than $\delta l_v(\gamma)$ gets cut off inside the square $S_1$, in which case we are done, otherwise a part greater than or equal to $2\delta l_v(\gamma)$ gets cut off outside $S_{1}\cup S_{2}$. If this segment gets cut off by the right edge of the square $S_{2}$ agaain then if a horizontal length $\delta l_v(\gamma)$ gets cut off inside the square $S_2$ we are again done, otherwise a portion at least $\delta l_v(\gamma)$ gets cut off outside $(S_1 \cup S_2)$. 

It is clear that segments can be cut off by a right edge at most twice, and in this case by hypothesis, the point returns to $R_1$ prior to reaching any other point within $S_1 \cup S_2$ or within a distance $D$ to the left of either $S_1$ or $S_2$. \footnote{which may both be arbitrarily close to the right edge} Then at least a segment of length $\delta l_v(\gamma)$ remains within one of the central square regions or touches the left edge of the particular central square. Note that these segments under consideration are subsets of the iterates of $I_1\cup I'_{2}$ under $H$.

If the segment touches both the right and left edges, we would be done with a complete h-segment within the square, otherwise from the argument in the previous two paragraphs, we would still have a segment of length at least $\delta l_{v}(\gamma)$ inserted within the square $S_1$ and we would also be done.

\bigskip

%Case(v):The situation is simpler when the second return of the rational point is to a point inside the square. In the case the return is to any point that is neither $R_1$ or $R_2$, the amount of horizontal length within the square is at least $\text{min}(d,l_{h}(I''_2\cup I_3))+\alpha l_v(I''_{2}\cup I_3)$ in which case we are done. If the return is to $R_1$ (or equivalently $R_2$) in that case again, we would have at least an amount $\text{min}(d,l_{h}(I''_2\cup I_3))+\alpha l_v(I''_{2}\cup I_3))$ to the right of the point $R_1$\footnote{the amount $d$ appears here since it may happen that $R_1(R_2)$ is almost at a distance $d$ to the left  of $LE_1(LE_2)$ }, or the segment is cut off by the right edge. In that case, the argument of the previous paragraph shows that a horizontal segment of length at least $\delta l_v(\gamma)$ is present to the left of the point $R_1$(or $R_2$).

\subsubsection{First return of $\gamma$ intersects both $S_1$, $S_2$}
 
 In the case the first return of $\gamma$ has intersection with both squares $S_1$ and $S_2$, as shown in \cref{fig:third},because of the symmetry, the distance between the right edge of $S_1$ and the left edge of $S_2$ is the same as the distance between the right edge of $S_2$ and the left edge of $S_1$. Suppose that the return was as shown in Figure 2, where the return has a segment $I_1$ in the square $S_1$, a segment $I_2$ in between the right edge of $S_1$ and the left edge of $S_2$, and finally a segment $S_3$ in the square $S_2$. The case where the return is between the right edge of $S_2$ and the left edge of $S_1$ is analogous.

In this case, it is enough that one of the following three holds,

\begin{align}
    l_h(I_1)\geq \delta l_v(\delta),\\
    l_{h}(I_3)\geq \delta l_v (\gamma) \\
    l_h (H(I_{2}))-l_{h}(I_2)\geq 2\delta l_v(\gamma), 
\end{align}

since this means that either one of the segments within the two squares are long enough, or that the horizontal length of image of the segment $I_{2}$ under the map $H$, increases by at least $2\delta l_{v}(\gamma)$ and so now at least one part of it has a segment that intersects either one of the squares with length at least $\delta l_{v}(\gamma)$.\footnote{Unlike for the usual toral linked twist map as considered in \cite{Prz}, it is not enough to only require an increase by an amount $\delta l_v(\gamma)$; it can happen that the image of $I_{2}$ goes from $RE_2$ to $LE_1$ in which case we would need an increment of at least $2\delta l_v(\gamma)$ to ensure a sufficient portion is pushed into either of the two squares.}

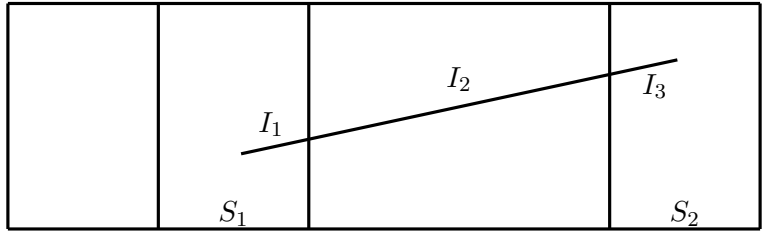
\begin{figure}
\centering
    \begin{tikzpicture}
        \draw[very thick] (-5,1.5) to (-5,-1.5);
        \draw[very thick] (5,1.5) to (5,-1.5);
        \draw[very thick] (5,1.5) to (-5,1.5);
        \draw[very thick] (5,-1.5) to (-5,-1.5);
        \draw[very thick] (3,-1.5) to (3,1.5);
        \draw[very thick] (-3,-1.5) to (-3,1.5);
        \draw[very thick] (-1,-1.5) to (-1,1.5);
        \node [below] at (-2,-1) {$S_1$};
        \node [below] at (4,-1) {$S_2$};
        \draw[very thick] (-1.9,-0.5) to (3.9,0.75);
        \node [above] at (-1.5,-0.4) {$I_1$};
        \node [above] at (1,0.2) {$I_2$};
        \node [below] at (3.6,0.7) {$I_3$};
        
    \end{tikzpicture} 
    \caption{The case where the return is to both the squares $S_1$ as well as $S_2$.}
    \label{fig:third}
\end{figure}

%For any one of Equations $(7-10)$ and any one of Equations $(11-13)$ to hold it would be enough for us to actually assume more simply that one of the following holds

\subsection{Showing the existence of a critical shear parameter:} We work with the Equations (7-13) and show the existence of a critical shear parameter $\alpha_0$ such that for all $\alpha\geq \alpha_0$ we have ergodicity for the linked shear map.

It would clearly be enough to show that all the following hold: 

\begin{align}
        d\geq 2\delta l_{v}(\gamma) \\
    \alpha l_{v} (I_{3}) \geq \delta l_{v}(\gamma). \\
    \alpha l_{v}(I_1)  +l_{h}(I_1) \geq 3\delta l_v(\gamma) \\
    l_{h}(I_3) \geq \delta l_v(\gamma) \\
    l_{h}(I_4) \geq \delta l_{v}(\gamma)
\end{align}

For equation (14) to hold, it is enough to show that 

\begin{align}
    \frac{\alpha l_{v}(I_{2})}{1+\alpha l_v(I_{2})}\geq 2\delta l_{v}(\gamma) \\
    \implies l_{v}(I_{2})\geq \frac{2\delta l_{v}(\gamma)}{\alpha(1-2\delta l_v(\gamma))}
\end{align}

We can always chose a $\delta>1$ such that $(1-2\delta l_{v}(\gamma))$ is positive; since the width of the strands of $\mathbb{T}$ are arbitrarily small and thus the vertical length $l_{v}(\gamma)$ is also arbitrarily small compared to $1/2$.

Further, from elementary geometry, we note that since the segment $I_2$ is within the cone $C$ or $C'$ depending on which square we are in, we must have that $l_{h}(I_{3})\geq l_{v}(I_{3})(\alpha+L)$. Thus for equation 13 and equation 15 to hold, it is enough to have that

\begin{align}
    l_{v}(I_{3})\geq \frac{\delta l_{v}(\gamma)}{\alpha+L} \ \Big( \geq \frac{\delta}{\alpha}l_{v}(\gamma) \Big)
\end{align}

(Note that clearly we have $\alpha> L$.)

Further, by an identical argument as above for the segment $I_1$, to satisfy equation 16 it is enough to have that 

\begin{align}
    l_{v}(I_{1})\geq \frac{3\delta l_{v}(\gamma)}{L+2\alpha}
\end{align}

Further, equation 18 is satisfied if we have  

\begin{align}
    l_{v}(I_{4})\geq \frac{\delta l_{v}(\gamma)}{\alpha+L}.
\end{align}

Thus we can find a $\delta>1$ satisfying the above relations, and either that segments $I_1,I_2,I_3$ satisfy equations (20-22), or that $I_4$ satisfies equation 23, if we ensure that 

\begin{align}
    l_{v}(\gamma) >l_{v}(\gamma)\Big( \frac{2}{\alpha+L}+\frac{3}{2\alpha+L}+\frac{2}{\alpha(1-2\l_{v}(\gamma))}\Big) 
\end{align}

Recall that $L=-\frac{\alpha}{2}+\sqrt{(\frac{\alpha}{2})^{2}-1}$.

We are precluding the possibility of having an $h$- segment after the segment $\gamma$ suffers just one shear. Since our lobes are symmetric and the length of $\mathbb{T}$ is unity, the width $w\ll L$, we can find some small enough $\eta$ such that $l_{v}(\gamma)(L+\alpha)<1/2+\eta$. Taking a crude estimate of $\eta=1/4$\footnote{depending on how small $w$ is in comparison with $L$, we can correspndingly take $\eta$ an order of magnitude smaller.}, we have an estimate of $l_{v}(\gamma)<\frac{3}{4(\alpha+L)}$. In that case, it will be enough to ensure that:

\begin{align}
    1 >\Big( \frac{2}{\alpha+L}+\frac{3}{2\alpha+L}+\frac{2}{\alpha(1-\frac{3}{2(\alpha+L)})}\Big),
\end{align}

in which case we could satisfy the estimate in equation 24. The equation above has a solution set of all shear parameters $\alpha>\alpha_1=6.23$.

In case we have to satisfy the set of equations 11 to 13, by arguments similar to ones used above, we aim to ensure:

\begin{align}
    l_{v}(I_{1})\geq \frac{\delta l_{v}(\gamma)}{(\alpha+L)}\\
    l_{v}(I_{3})\geq \frac{\delta l_{v}(\gamma)}{(\alpha+L)}\\
    l_{v}(I_{2})\geq \frac{2\delta l_{v}(\gamma)}{\alpha}
\end{align}

It is enough to ensure the following:

\begin{align}
    l_{v}(\gamma)=\sum\limits_{i=1}^{3}l_{v}(I_{i})\geq l_{v}(\gamma)\Big( \frac{2}{\alpha +L}+\frac{2}{\alpha}\Big),
\end{align}

i.e.

\begin{align}
    1>\frac{2}{\alpha+L}+\frac{2}{\alpha}.
\end{align}

The above is also ensured for all $\alpha> \alpha_2=4.13$. 

Thus combining the two cases above, we get an optimal constant of $6.23$, and ergodicity and the Bernoulli property is established for all $\alpha>\alpha_0 =6.23$.

\section{Acknowledgements:} The author is thankful to Boris Hasselblatt and Curtis Heberle for discussions on this problem, and also to Feliks Przytycki for useful feedback on this question.

\bigskip

\appendix
\section{\small{Variation of hyperbolic behavior with the angle between the two lobes}} 

In the surgery process, it is also natural to consider the lobes of the surgery track to be aligned so that they overlap not necessarily perpendicularly as was assumed throughout the analysis in this paper, but at a certain angle other than $\pi/2$.

For simplicity we only work with a constant shear of $A>2$. We show hyperbolic behavior when the angle $\alpha$ between the two lobes approaches $\pi/2$. For small values of $\alpha$ there is a cut-off value below which we see elliptic behavior as shown in the lemma below.

We revert to a coordinate system at the angle $\alpha$ to the $x-$axis, apply the shear

$$\begin{pmatrix}1& A \\ 0& 1\end{pmatrix},$$

then revert back to the original coordinate system.\footnote{In the process, we are neglecting a possible flip of coordinates, which is not a problem since eventually we will multiply the rotation matrix twice and thus for this rotation we would eventually multiply by $(-1)^2$.}

\begin{figure}[h]
\centering
\includegraphics[width=0.5\textwidth]{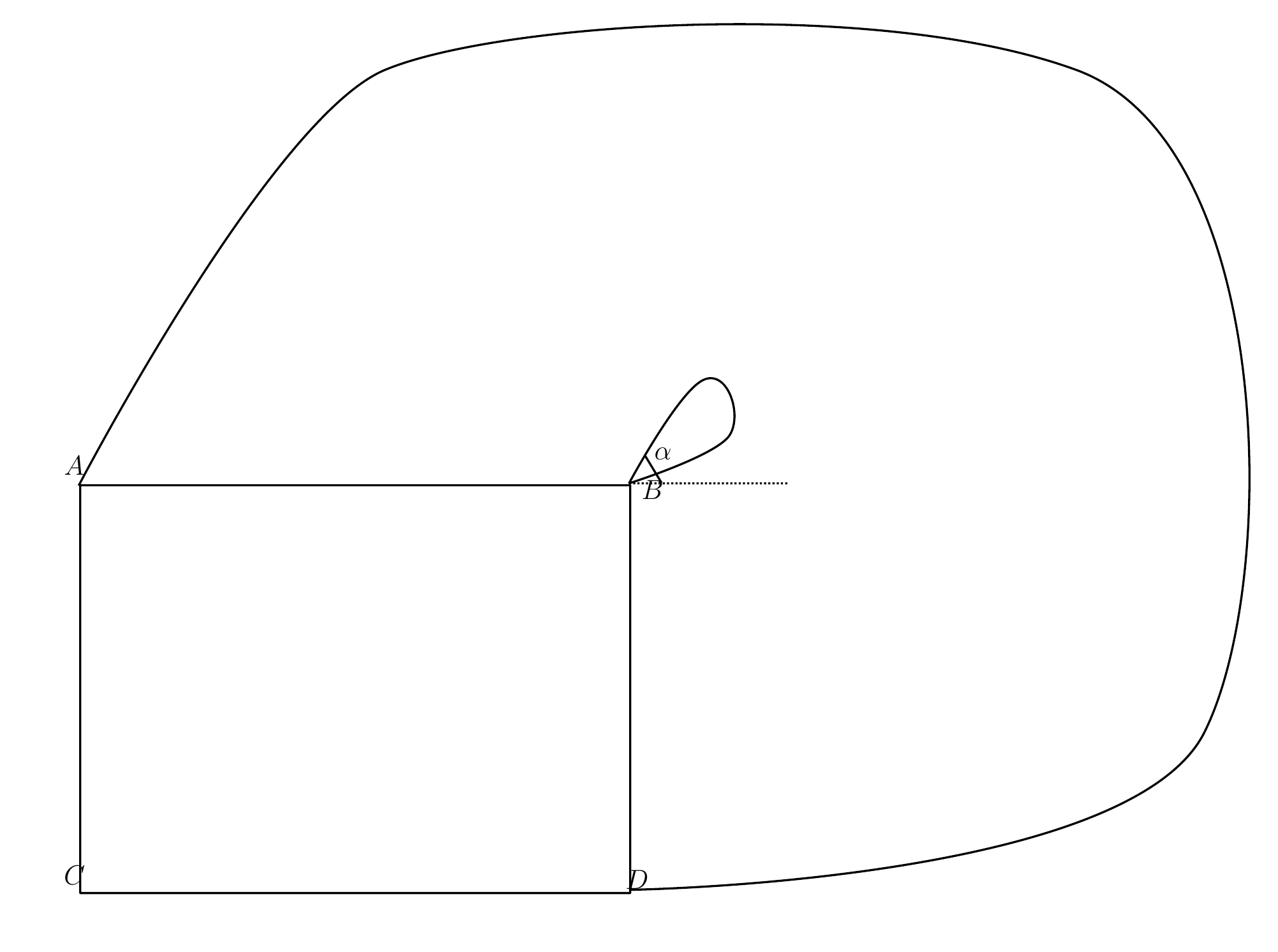}
%\huge\input{drawing.pdf_tex}
\caption{The angle $\alpha$ between the planar projections of the two lobes varies between $0$ and $\pi/2$.}
\label{fig:inksone3}
\end{figure}

The effect of the two shears  is thus given by:

\begin{align}
    \begin{pmatrix}
        \cos \alpha & \sin \alpha
\\ -\sin \alpha & \cos \alpha    \end{pmatrix}\begin{pmatrix}1& A \\ 0& 1\end{pmatrix}\begin{pmatrix}
        \cos \alpha & -\sin \alpha
\\ \sin \alpha & \cos \alpha  \end{pmatrix}\begin{pmatrix}1& A \\ 0& 1\end{pmatrix}\\
=\begin{pmatrix} 1+A\sin \alpha \cos \alpha & A+A\cos^2 \alpha +A^2\sin\alpha\cos\alpha\\ -A\sin^{2}\alpha & 1-A^2\sin^2 \alpha-A\sin \alpha \cos \alpha \end{pmatrix}
\end{align}

\begin{lemma}\label{lem:eig}
Let
\[
M=\begin{pmatrix}1+A\sin\alpha\cos\alpha & A+A\cos^{2}\alpha+A^{2}\sin\alpha\cos\alpha\\[2pt] -A\sin^{2}\alpha & 1-A^{2}\sin^{2}\alpha-A\sin\alpha\cos\alpha\end{pmatrix}.
\]
Then $\det M=1$ and $\operatorname{tr}M=2-A^{2}\sin^{2}\alpha$, so the eigenvalues are
\[
\lambda_{\pm}=1-\tfrac12 A^{2}\sin^{2}\alpha\;\pm\;\tfrac{i}{2}\,|A\sin\alpha|\sqrt{4-A^{2}\sin^{2}\alpha}.
\]
If $|A\sin\alpha|<2$, then $\lambda_{\pm}=e^{\pm i\theta}$ with
\[
\theta=2\arcsin\!\Big(\tfrac{|A\sin\alpha|}{2}\Big),
\]
and $\theta=|A|\,\alpha+O(\alpha^{3})$ as $\alpha\to0$.
If $|A\sin\alpha|=2$, then $\lambda_{\pm}=-1$; if $|A\sin\alpha|>2$, then $\lambda_{\pm}$ are real, negative, and $\lambda_{+}\lambda_{-}=1$. Thus in particular, when $\alpha=\pi/2$, as considered in this paper, we see the eigenvalues are real and we get hyperbolic behavior.
\end{lemma}

\end{document}